\documentclass[11pt,reqno]{amsart}
\usepackage{amsfonts}
\usepackage{amsfonts,amssymb,amsmath}
\usepackage{float}

\usepackage[latin1]{inputenc}
\usepackage{color}
\usepackage{colortbl}
\usepackage{graphicx}
 \usepackage[usenames,dvipsnames]{pstricks}
 \usepackage{epsfig}
 \usepackage{pst-grad} % For gradients
 \usepackage{pst-plot} % For axes

\newtheorem{theorem}{Theorem}[section]
\newtheorem{lemma}[theorem]{Lemma}
\newtheorem{proposition}[theorem]{Proposition}
\newtheorem{corollary}[theorem]{Corollary}
\newtheorem{remark}[theorem]{Remark}
\newtheorem{definition}{Definition}[section]

\newcommand{\bb}[1]{{\mathbb #1}}

\newcommand{\eps}{\varepsilon}
\newcommand{\ve}{\varepsilon}

\newcommand{\p}{\partial}
\newcommand{\pfrac}[2]{\genfrac{}{}{}{1}{#1}{#2}}

\usepackage{a4wide}

\def\P{{\mathbb P}}
\def\R{{\mathbb R}}

\def\T{{\mathbb{T}}}

\def\ve{{\varepsilon}}

\def\T{\bb T}
\def\Tm{T_{\rm max}}
\def\TmN{T^N_{\rm max}}

\title{{\sc A particle system with explosions: law of large numbers\\ for the density of particles and the blow-up time.}}

\keywords{Hydrodynamic limit, Parabolic equations, blow-up}
\subjclass[2010]{60K35, 35K55, 35B40}
\date{}

\author{Tertuliano Franco}
\address{Universidade Federal da Bahia, Salvador, Brazil}
\email{tertu@impa.br}
\author{Pablo Groisman}
\address{Departamento de Matem\'atica, Fac. Cs. Exactas y Naturales,
UBA and IMAS-CONICET, Buenos Aires, Argentina}
\email{pgroisma@dm.uba.ar}

\begin{document}

\begin{abstract}
Consider a system of independent random walks in the discrete torus with
creation-annihilation of particles and possible explosion of the
total number of particles in finite time. 
Rescaling space and rates for diffusion/creation/annihilation of particles, we
obtain a stong law of large numbers for the density of particles in the
supremum norm. The limiting object is a classical solution to the semilinear
heat equation $\p_t u=\p_{xx} u + f(u)$. If $f(u)=u^p$, $1<p\leq 3$,  we also obtain a law of large numbers for the explosion time.
\end{abstract}

\maketitle

\section{Introduction}
\setcounter{equation}{0}

We consider nearest-neighbors symmetric independent random walks
superposed with birth and death dynamics in the discrete torus. 
At rate one, each particle jumps to one of its neighbors with the same
probability. In addition,  if at a site there are $r$ particles, at that site a
new particle is created at rate $b(r)$ and a
particle is destroyed at rate $d(r)$.

We study suitable scaled versions of this process, that (as will be shown) converge almost surely
in the $L^\infty$-norm to the solution of the semilinear parabolic problem
\begin{equation}
\label{PDE}
\left\{\begin{array}{ll}
u_t = u_{xx} + f(u)  \qquad & (x,t) \in \bb T \times [0,T),\\
 u(x,0) = \varphi (x) \ge 0 \qquad & x \in \bb T\,. \\
\end{array}\right.
\end{equation}
where $f=b-d$ is assumed to be smooth, $\varphi$ is smooth and nonnegative, and $\bb T$ denotes the continuous one-dimensional torus.
This equation has been widely studied in the literature, being used to
model
diverse processes in mechanics, physics, chemistry, technology, biology and
many other areas. For instance, under certain conditions, it describes
conduction in plasma, 
gas filtration and liquids in porous media, chemical reactions, processes of growth and migration of populations, etc.

One of the most remarkable properties of this equation is the possible
occurrence of singularities due to the presence of the nonlinear source $f$.
Even starting from regular data, for which there exist an existence, 
uniqueness and continuous dependence theory for short times, the solution may
develop singularities in finite time. Although for linear evolution problems
singularities may occur, they appear due to singularities in the coefficients or
in the problem data, while in this case, singularities appear because of the
nonlinear essence of the equation, and the time and space localization of them has to
be determined through a careful analysis.

In this problem, singularities appear in the simplest way: they are due to a
fast increasing of the solution that leads the $L^\infty$ norm to grow-up to
infinity in a finite time $T_{\textrm{max}} = T_{\textrm{max}}(\varphi)$.

The phenomenon is known as {\em blow-up}, and is interpreted as an
abrupt change in the order of magnitude in the modeled quantity. It was
successfully used to model, for instance, explosions in exothermic chemical
reactions, population dynamics, fatigue cracking (in this case explosion means
that a microscopic crack changes its scale and becomes  macroscopic,
indicating a  crack in the material due to fatigue).

In view of this, it is important to understand the microscopic behavior of this
kind of systems, and in particular their singularities (presence of them,
space-time location, order, etc.).

A well known condition on the nonlinear term $f$ that assures the existence of
solutions with blow-up is being convex, strictly positive in some interval $[a,
+\infty)$ and 
\begin{equation}\label{blow.up.condition}
\int^\infty_a \frac{ds}{f(s)} < \infty\,.
\end{equation}

The most simple source verifying this conditions is $f(s)=s^{p}$,
$p >1$. For a general description of the blow-up problem we refer the
reader to the books \cite{QS, SGKM}
and the surveys \cite{BB,GV}.

Coming back to the Markov chain (the particle system), the scaling here is in
the same spirit of \cite{BlouComp1991,BlouLawo1992,KoteLawo1986, KoteHigh1988}. 
The initial quantity of particles per site is also rescaled, different than in
the hydrodynamic limit context \cite{KipnScal1999}. 

As main results, we obtain  almost sure convergence for the density of particles
in the supremum norm for any compact time interval not
containing $T_{\textrm{max}}$.

This result was obtained by Blount, \cite{BlouLawo1992} for
$f$ a polynomial with negative leading term. In that case solutions are known to be bounded for every $t\ge 0$ and hence globally defined, see also \cite{AT, BlouComp1991,
KoteLawo1986, KoteHigh1988}.

The first part of the proof consists in proving the result for birth rates $b$ with compact support, where there is no blow-up and the
solution is bounded. This follows essentially the
work of Blount \cite{BlouLawo1992}, the main difference being that we consider
any continuously differentiable $f$ instead of polynomials. In the second part of the proof,
we couple a sequence of Markov chains as the one described above, with birth rates
$b_j$, where $(b_j)_{j\ge 1}$ is a sequence of smooth functions
with compact support approximating $b$. This coupling allows us to extend the
result proved in the first part to any smooth $b$, including those one
satisfying \eqref{blow.up.condition}.

As an immediate corollary, the liminf  of the explosion times of the
discrete systems is no smaller than the blow-up time of the solution to the PDE
\eqref{PDE}. The opposite inequality  is much harder and we are only able to
obtain it in some specific cases.
Assuming that $f(s)$ verifies \eqref{blow.up.condition} and
that $d$ is a bounded or linear function, we prove
that, for each $N\ge 1$, the corresponding particle system explodes with
probability one. Under the additional condition $f(s)=b(s)-d(s)=s^p$, with
$1 < p \leq 3$ we also prove that the explosion time of the particle system
converges in probability, as $N\to \infty$, to the
blow-up time of the solution to the PDE. 
 
We remark here that if $f$ is globally Lipschitz then there is no explosion
with probability one.

To the best of our knowledge, there is no previous work about the limit density of
interacting particle systems exhibiting blow-up. For instance, 
\cite{BlouComp1991,BlouLawo1992,KoteLawo1986, KoteHigh1988} in this type of
scaling, and \cite{Mourragui} in the hydrodynamical limit context, have considered creation
of particles, in some cases with unbounded limit, but all of them deal with processes defined for all times.

It is worth to notice that since the particle system explodes in finite
time, the expectation of the number of particles is infinity at any positive
time. Hence, any method based on expectations is doomed to fail. This motivates
the use of couplings.

\section{Notation and results}\label{s.2}

\subsection{The particle system} Denote $\T_N=\bb Z/(N\bb Z)$ the discrete torus with $N$ points.
Fix two nonnegative smooth functions $b,d:\bb R_+\to\bb R_+$ such that $d(0)=0$.  Consider also a parameter $\ell\in \bb N$, which will represent
the number of particles per site in the initial configuration. 
We characterize the  continuous time Markov chain 
$$(\eta(t))_{t\geq 0}=(\eta_1(t),\ldots,\eta_N(t))_{t\geq 0}$$
 with state space $\Omega_N= \bb N^{\T_N}\cup \{\infty\}$ by its jump rates given by 
\begin{enumerate}
 \item[\textbullet] at rate $N^2\eta_k$, a particle jumps from $k$ to $k+1$; 
 \item[\textbullet] at rate $N^2\eta_k$, a particle jumps from $k$ to $k-1$; 
\item[\textbullet] at rate $\ell b(\ell^{-1}\eta_k)$, a new particle is created at $k$; 
\item[\textbullet] at rate $\ell d(\ell^{-1}\eta_k)$, a particle is destroyed at $k$. 
\end{enumerate}
The transitions above are assumed for all $k\in \T_N$.
Aiming not carry on the notation, we do not index $\eta(t)$ on $N$ and on $\ell$. Since there are no assumptions on the behavior of $b$ at infinity, the waiting times of this Markov chain can be summable. If that is the case we say that the process \emph{explodes} or \emph{blows up}, and we define the state of the process as $\infty$ for times greater or equal than the sum of the waiting times, that we call $\TmN$. More precisely, define
\[
 T^N_M:=\inf\{t\ge 0\colon \|\eta(t)\|_\infty\ge \ell M \} \mbox{ and }
\TmN:=\lim_{M\to \infty}  T^N_M\,.
 \]
Hence we can easily define $\eta(t)$ for $t<\TmN$ and we define $\eta(t)=\infty$ for $t\ge \TmN$. A graphical construction of this process is given in Section \ref{graphical.construction}. For more on explosions of Markov chains we refer to \cite{Norris}. 

Next, we define the spatial density of particles $X^N$, which is a function defined on the continuum torus $\bb T=[0,1]$, identifying $0$ and $1$ and considering a fixed  orientation. 
 
For $k\in \T_N$, let $x_k=k/N$ and define
\begin{equation*}
 X^N(x_k,t) \;=\; \ell^{-1} \eta_k(t)\,.
\end{equation*}
Also for $x_k<x<x_{k+1}$, we define the density of particles
by linear interpolation, i.e.
\begin{equation*}
 X^N(x,t) \;=\; (Nx-k) X^N(x_{k+1},t) + (k+1-Nx) X^N(x_k,t)\,.
\end{equation*}
If $\eta(t)=\infty$, we say that $\Vert X^N(\cdot, t)\Vert_\infty=\infty$ as well.
We point out that this interpolation has no special meaning. Using instead a smoother interpolation or even defining $X^N$ as a step function would not change the results.

\subsection{The partial differential equation}\label{sub2.2} We make the
following assumptions on the data of problem \eqref{PDE}:

\begin{enumerate}
\item[\textbullet]\label{a2} The initial datum $\varphi$ is $C^4$
and nonnegative.
\item[\textbullet]\label{a3} The source term $f$ is $C^1$ and $f(0)\ge0$.
\end{enumerate}
Under the above assumptions this equation has a unique (local) solution $u$,
which is smooth in some interval $(0,T_{\textrm{max}})$. Here $T_{\textrm{max}}$
is the maximal existence time. If $f$ is globally Lipschitz then
$T_{\textrm{max}}=+\infty$
(global existence) but it can be proved (see Theorem \ref{explosion_PDE} in Section \ref{asymp.behavior}) that if $f$ is convex and verifies
\eqref{blow.up.condition} then, for positive $\varphi$, the solution blows up in finite
time, meaning that it is smooth in $(0,T_{\textrm{max}})$, but
$$
\lim_{t\nearrow T_{\textrm{max}}} \|u(\cdot,t)\|_{\infty} = \infty\,.
$$
For general references on the study of this equation, existence, uniqueness and
asymptotic behavior (including the blow-up case) see the books \cite{Pao, QS, SGKM, JLV}.

The partial differential equation \eqref{PDE} satisfies a comparison principle\footnote{This is clarified in Section \ref{s.3}.}. Since we require $f(0)\ge0$, $u\equiv 0$ is a sub-solution to this problem and hence for nonnegative initial data, the solution is positive.
All the regularity assumptions on the data of the problem are to guarantee the smoothness of the solution. They can be relaxed to some extent, but since we are not focused on the 
problems arising due to the the lack of regularity of the solutions, we prefer not to include them to simplify the exposition.
\medskip

We are in position to state the main results of this paper.

\begin{theorem}\label{t2.1}
 Assume that
\begin{enumerate}
 \item[(A1)] $\Vert X^N(\cdot,0)-\varphi(\cdot)\Vert_\infty \to 0$ almost
surely;
 \item[(A2)] for any  $c>0$, $\ell=\ell(N)$ satisfies $\sum_{N\geq 0} N^3
e^{-c\, \ell}<\infty\,$.
\end{enumerate}
Then, for any $T\in [0,\Tm)$,
\begin{equation}\label{eq2}
\lim_{N\to\infty} \sup_{t\in[0,T]} \Vert  X^N(\cdot,t)-u(\cdot,t)\Vert_\infty\;=\; 0\,,\quad \textrm{ almost surely.}
\end{equation}
\end{theorem}
\begin{remark}
By {\rm A1}, the parameter $\ell$ represents the order of the initial quantity
of particles per site. Condition {\rm A2} states that the growth of $\ell$
cannot be too slow in comparison with $N$. For instance,  $\ell(N)=N^\eps$
satisfies {\rm A2} for any $\eps>0$.
\end{remark}
An immediate corollary of Theorem \ref{t2.1} is the following
\begin{corollary}
\label{cor.liminf}
If {\rm A1} and {\rm A2} hold, we have
\[
 \liminf_{N\to \infty} \TmN \ge \Tm\, \quad \mbox{almost surely}.
\]
\end{corollary}
The left hand side in the above equation can be infinity in general. Next we find conditions to guarantee that that is not the case.

\begin{proposition}
\label{bu.prob.1}
Let $b \in C^1$ be convex, positive and such that $\int_0^\infty 1/b(s)\, ds < +\infty$. Assume also that $d$ is bounded or linear, then for every $N$ we have
\[
 \P(\Tm^N < +\infty)=1\,.
\]
\end{proposition}

Conditions on the growth of $b$ cannot be removed as can be shown with simple
examples. The convexity assumption is technical, but we are not able to remove
it. It can be weakened to some extent assuming that $b$ is convex on some
interval of the form $[a,+\infty)$.
Finally, we obtain
\begin{theorem}\label{t2.2} Assume $d$ is bounded or linear and $b(s)=s^p+d(s)$
with $1<p\le 3$. If {\rm A1} and {\rm A2} hold, then
\begin{equation}\label{lim_prob}
\lim_{N\to\infty }\TmN=
\Tm\,,\quad\textrm{in probability.} 
\end{equation}
\end{theorem}

The paper is organized as follows: in Section \ref{graphical.construction} we
give a graphical construction of the process. Section \ref{s.3} deals with a
semidiscrete approximation of equation \eqref{PDE}. We discretize the space
variable and prove that the solution of the ODE obtained with this procedure
converges to $u$ as the mesh parameter goes to zero. In Section
\ref{limit.density.of.particles} we prove Theorem \ref{t2.1}. We first prove
this theorem for birth rates $b$ with compact support relying on results of
Blount \cite{BlouLawo1992} and then we extend these results to the general case
that allows explosions by means of truncation and a coupling argument. Finally,
in Section \ref{asymp.behavior} we study the asymptotic behavior of the particle
system. We prove that the process explodes with probability one, Proposition
\ref{bu.prob.1} and that the blow-up times converges to the explosion time of
the PDE \eqref{PDE}, Theorem \ref{t2.2}.

\section{Graphical construction}
\label{graphical.construction}
In this section we give the so called Harris graphical construction of the
process: we construct the particle system as a deterministic function of a
family of Poisson processes. This construction will be useful on the one
hand to couple processes with different birth rates $b$ and on the other
hand to construct jointly a birth and death process that bounds the total
number of particles in the system from below.

\def\Nr{\mathcal N}
\def\Nb{\mathcal N_b}
\def\Nd{\mathcal N_d}
\def\one{{\bf 1}}

Let $(\Nr^+(i), \,\Nr^-(i), \, 1\ge i)$ be a family of Poisson processes in $\R_+$ with rate $N^2$. Let also $(\Nb(r,k),\, \Nd(r,k), \,  r\ge 1, 1\le k \le N)$ be a family of Poisson processes in $\R_+$ with rates $\ell b (\ell^{-1} r)$, $\ell d (\ell^{-1} r)$ respectively. All the processes are taken independent. We construct a process $\xi(t) = (\xi_1(t),\xi_2(t) \dots)$ that determines the position of each particle and the process $\eta(t)$ is defined as the empirical measure of $\xi(t)$, i.e.
\[
 \eta_k(t)=\sum_{i= 1}^{K(t)}\one\{\xi_i(t)=k\}\,.
\]
The variable $K(t)$ is the total number of particles in the system and will be defined inductively jointly with the construction of the process. Initially we start with $K(0)$ labeled particles $1, \dots, K(0)$ at positions $\xi_1, \dots, \xi_{K(0)}$. Assume the process is defined up to time $s\ge 0$ and proceed by recurrence. Start with $s=0$ and define

\[
 \tau(s, \xi(s)):=\inf\left\{t>s \colon t\in \cup_{i= 1}^{K(s)} (\Nr^+(i)
\cup \Nr^-(i))\cup
\cup_{k\in \T_N} \Nb(\eta_k(s),k) \cup \Nd(\eta_k(s),k) \right\}.
\]
For $t\in [s,\tau)$ define $\xi(t)=\xi(s)$ and then
\begin{enumerate}
 \item[1.] If $\tau \in \Nr^\pm(i)$ define $\xi_i(\tau)=\xi_i(s) \pm 1$ and $\xi_j(\tau)=\xi_j(s)$ if $j \ne i$. 
 \item[2.] If $\tau \in \Nb(\eta_k(s))$ set $\xi_{K(s)+1}=k$ and $K(\tau)=K(s)+1$.
 \item[3.] If $\tau \in \Nd(\eta_k(s))$, let $j:=\min\{i\colon \xi_i(s)=k\}$. Set $\xi_j(\tau)=\xi_{K(s)}(s)$  and $K(\tau)=K(s)-1$.
\end{enumerate}
In words, if $\tau \in  (\Nr^+(i)
\cup \Nr^-(i))$ (where $i$ is a particle in the system),
then this particle moves to the right or to the left according to wether $\tau
\in \Nr^+(i)$ or $\tau \in \Nr^-(i)$ . If $\tau \in \Nb(\eta_k(s))$, a new
particle is created at site $k$, and hence $K(s)$ increases in one. Finally, if
$\tau \in \Nd(\eta_k(s))$, a particle is killed at site $k$. We kill the
particle with minimum index and we assign this index to the particle with index
$K(s)$ (and then decrease $K(s)$ in one), so that the alive particles are alway
the ones with index $1, \dots, K(\cdot)$.
The process is then defined up to time $\tau$. Put $s=\tau$ and iterate to
define $\xi(t)$ and $\eta(t)$ up to time $t=\Tm^N$. Observe that $\Tm^N<\infty$
if and only if $\eta(\Tm^N):=\lim_{t\nearrow  \Tm^N}
\|\eta(t)\|_\infty=\infty$, and in this case, the sum of the waiting times is
summable and equal to $\Tm^N$.

\bigskip

\paragraph{{\bf Coupling processes with different birth rates}} Observe that if
we want to construct two different copies $\eta, \tilde \eta$ of the above
process with two different birth rates $b$, $\tilde b$ and we have $b(s)=\tilde
b(s)$ for $s\le M$ then we can use the same Poisson processes $\Nr^+(i),
\,\Nr^-(i), \, 1\ge i$,  $ \Nd(r, \cdot), \,  r\ge 1$  and $\Nb(r,\cdot)$ for $r
\le M$. In this sense we obtain that almost surely  $\eta(t)=\tilde \eta(t)$ for
$0\le t \le T^N_M$, the first time that the process reaches the value $M$.

\section{Convergence of a semidiscrete scheme}\label{s.3}

We now consider deterministic spatial discretizations of
\eqref{PDE}, keeping continuous the time variable.
The goal is to prove convergence of such spatial discretizations to the solution $u$ of the partial differential 
equation \eqref{PDE}. This result will be used as an intermediate step in the proof of the Theorem \ref{t2.1}.  
% Denote by $U (t)=(
% u_{1}(t),...., u_{N}(t))$ the values of
% the approximation at the nodes
% $x_i= -L + (i-1) h$ at time $t$, where  .
% Namely,  $U (t)$ is defined by means of the following equation:
% \begin{equation*}  
% \left\{\begin{array}{l}
% U'(t)= - A g(U) (t) + f(U)(t), \\
% u_{1} (t) = u_{N} (t)=0, \\
% U(0)= \varphi^I,
% \end{array} \right.
% \end{equation*}
% where $A$ is the discrete Laplacian and $\varphi^I$ is the
% Lagrange interpolation of the initial datum, $\varphi$.
% Writing this equation explicitly we obtain the following ODE
% system: 

Throughout this section, we assume that the function $f$ in \eqref{PDE} is globally Lipschitz.

We define the semidiscrete
approximation $u^N(t)=(u^N_1(t),\ldots,u^N_N(t))$ of the PDE \eqref{PDE}  as the
solution of the following ODE system:
\begin{equation}
\left\{
\begin{array}{ll}
\label{system}
\frac{d}{dt}u^N_k(t) =   \displaystyle{N^{2}} [ u^N_{k+1}(t)
-2u^N_k(t)+u^N_{k-1}(t)] + f(u^N_k(t)),  &k\in \T_N, \\
\\
u^N_k (0)=  \varphi (x_k), & k\in \T_N.
\end{array}
\right.
\end{equation}

\begin{proposition}\label{convergence.semidiscrete}
Let $u\in C^{4,1}(\T\times[0,T])$ be a positive
solution of {\rm (\ref{PDE})} and $u^N(t)$ the semidiscrete
approximation given by {\rm (\ref{system})}. Then, there exists a
positive constant $C$ depending on the $C^{4,1}(\T\times[0,T])$
norm of $u$ such that, for every $N$ large enough, 
$$
\sup_{t\in[0,T]}\max_{k\in \T_N}|u(x_k,t)-u^N_k(t)|\;\le\; C N^{-2}. $$
\end{proposition}
\noindent As a consequence of Proposition \ref{convergence.semidiscrete}, 
\begin{equation}\label{bound_u}
\limsup_{N\to \infty} \sup_{t\in[0,T]}\max_{k\in \T_N} \vert u^N_k(t)\vert\;<\;
\infty\,.
\end{equation}

We need the next lemma about solutions (and supersolutions) of the  following ODE system
% To end the proof of Proposition \ref{convergence.semidiscrete} we need to study solutions (and supersolutions) of the problem
% supersolutions) 
\begin{equation}
\label{eq.error}
 \begin{cases}
 z'_k   =  N^{2}(z_{k+1}-2z_k +z_{i-1} ) + C_*(|z_k| + N^{-2} ), &  k\in \T_N,
\\
z_k(0)=0, &   k\in \T_N.
 \end{cases}
\end{equation}
\begin{definition}\label{def.super} We say that $\overline{Z}=(z_{1},\dots, z_N)$ is  a
supersolution of \eqref{eq.error} if 
\begin{equation}
\label{super}
 \begin{cases}
 \bar z'_k   \ge  N^{2}(\bar z_{k+1}
-2\bar z_k +\bar z_{k-1} ) + C_*(|\bar z_k| + N^{-2})\,,  &   k\in \T_N, \\
z_k(0)\ge0, &   k\in \T_N.
 \end{cases}
\end{equation}
Analogously, we say that $\underline{Z}$ is a subsolution if
it satisfies (\ref{super}) with the reverse inequalities.
\end{definition}
% \medskip

\begin{lemma}
\label{compar}
Let $\overline{Z}$ and $\underline{Z}$ be a supersolution and a subsolution of \eqref{super} respectively, and let $Z$ be a solution of \eqref{eq.error}. Then
$$ \overline{Z} (t) \ge Z(t) \ge \underline{Z} (t)\,. $$
\end{lemma}

\begin{proof} By an approximation procedure we restrict ourselves to
consider strict inequalities in \eqref{super}. If that is not the case, we consider 
$\tilde Z(t)=\overline{Z}(t) + C\ve t$ with adequate $C$, and letting $\ve \searrow 0$ yields the result.

We prove that $\overline{Z} (t) > Z (t)$ arguing by contradiction. Assume that there exists a
first time $t_*$ and  $k\in \T_N$ such that $\overline{z}_k (t_*)= z_k (t_*)$. 
Then, we would have
$$0 \ge \overline{z}_k '(t_*)- z_k '(t_*) >
 N^{2}(\overline{z}_{k+1}(t_*)- z_{k+1} (t_*)
 +\overline{z}_{k-1}(t_*)- z_{k-1} (t_*) ) \ge 0\,,$$
 a contradiction. The inequality $Z(t) \ge \underline{Z}(t)$ is handled in a
similar way.
\end{proof}

\begin{proof}[\textbf{Proof of Proposition \ref{convergence.semidiscrete}}]
% In the course of this proof we will denote by $C$ a positive constant independent
% of $h$ which can change from one inequality to another. In contrast, the positive constant $C_*$ will
% not change.
% 
% \medskip

For $k\in \T_N$ denote $u_k= u (x_k,t)$ and define the error function 
$$e_k:=u^N_k-u_k\,.$$
By means of Taylor's expansion, for $k\in
\T_N$, there
exist $c_k\in (x_{k}, x_{k+1})$ and 
$\tilde{c}_k\in (x_{k-1}, x_{k})$ such that
\begin{equation*}
 u_{k+1}=u_{k}+u_x(x_k,t)\frac{1}{N}+u_{xx}(x_k,t)\frac{1}{2!N^2}
+u_{xxx}(x_k,t)\frac{1}{3!N^3}+u_{xxxx}(c_k,t)\frac{1}{4!N^4}
 \end{equation*}
and
\begin{equation*}
 u_{k-1}=u_{k}-u_x(x_k,t)\frac{1}{N}+u_{xx}(x_k,t)\frac{1}{2!N^2}
-u_{xxx}(x_k,t)\frac{1}{3!N^3}+u_{xxxx}(\tilde{c}_k,t)\frac{1}{4!N^4}\,.
 \end{equation*}
% \begin{equation*}
%  g(v_{i-1})=g(v_{i})-(g\circ u)_x(x_i)h+(g\circ u)_{xx}(x_i)\frac{h^2}{2!}-(g\circ u)_{xxx}(x_i)\frac{h^3}{3!}+(g\circ u)_{xxxx}(\tilde{c}_i)\frac{h^4}{4!}\,.
%  \end{equation*}
Summing the equations above and recalling that $u$ is the solution of \eqref{PDE}
% $a_i:=\frac{1}{4!}((g\circ u)_{xxxx}(\tilde{c}_{i})+ (g\circ u)_{xxxx}(c_{i}))$ 
gives 
% Summing the Taylor approximations of $g(v_{i+1})$ and $g(v_{i-1})$ at $x_i$ we get that for $0\le i \le N$, there exist $c_i \in (x_i, x_{i+1})$ and 
% $\tilde{c}_i \in (x_{i-1}, x_{i})$ such that $v_i$ verifies
$$
u_k'=N^{2}(u_{k+1}-2u_{k} + u_{k-1}) + f(u_k) -
\frac{1}{4!N^2}( u_{xxxx}(\tilde{c}_{k})+ u_{xxxx}(c_{k})).
$$
Writing $a_k:=\frac{1}{4!}(u_{xxxx}(\tilde{c}_{k})+ u_{xxxx}(c_{k}))$, we get that the error function satisfies, for
$k\in \T_N$, 
\begin{equation*}\label{eq06}
e_k'\;=\; N^2(e_{k+1}-2 e_k + e_{k-1}) + f(u^N_k)-f(u_k) -a_kN^{-2}\,.
\end{equation*}
Since $f$ is globally Lipschitz,
there exists a positive constant $C_*$ independent of $N$ such that
\[
e_k'\le   N^{2}(e_{k+1} -2e_k +e_{k-1})+C_*(|e_k|  + N^{-2})\,, \quad \forall
k\in \T_N\,.
\]
Hence $(e_1,\dots, e_N)$ is a sub-solution of \eqref{super}. Consider the super-solution $\bar Z=(\bar z_{1}, \dots,
\bar z_N)$ given by $\bar z_k(t)=e^{2C_* t}/N^2$. Notice that $\bar Z$
verifies \eqref{super}. By Lemma \ref{compar},
$$
e_i (t)\le \bar z_i(t)\le e^{2C_* T}/N^2 \quad \mbox{for all } k \in \T_N\,.
$$
Repeating the same arguments as before with $-e_i$, we obtain
$$
|e_k(t)|\le \bar z_k(t)\le e^{2C_* T}/N^2 \quad \mbox{for all } k \in \T_N\,.
$$
This completes the proof.
\end{proof}

\medskip

\section{Limit for the density of particles.}
\label{limit.density.of.particles}
The following key estimate is obtained by Blount in \cite{BlouLawo1992}.
\begin{theorem}[Blount, \cite{BlouLawo1992}] Assume $b$ is Lipschitz continuous
with compact support. Then there exist constants $K,a>0$ depending on $T$ and
$\varphi$ and a process $Y^N$ such that
\begin{equation*}
 \Vert \overline{X}^N(t)-u^N(t)\Vert_\infty\;\leq\; (1+Kt\,e^{Kt})\,\Big(K\Vert \overline{X}^N(0)-u^N(0)\Vert_\infty+\sup_{s\in[0,t]}\Vert Y^N(s)\Vert_\infty\Big)\,.
\end{equation*}
Moreover,
\begin{equation}
\label{BC}
\P \big(e^{-4T}\sup_{t\in[0,T]}\Vert Y^N(t)\Vert_\infty >\ve \big) \le
4N^3e^{-a\ve^2
\ell}\,.
\end{equation}
\end{theorem}

The proof in \cite{BlouLawo1992} considers the case where $f$ is a polynomial with negative leading term, but the proof can be extended to our case with no difficulty.

\begin{proof}[Proof of Theorem \ref{t2.1}]
Assume first that $b$ has compact support. Condition A1 in Theorem
\ref{t2.1} means that the bound in \eqref{BC} is summable in $N$, so
Borel-Cantelli's lemma implies $\sup_{t\in[0,T]}\Vert
\overline{X}^N(t)-u^N(t)\Vert_\infty \to 0$ almost surely. This fact, combined
with Proposition \ref{convergence.semidiscrete} gives us
\[
\lim_{N\to \infty}\sup_{t\in[0,T]}\Vert \overline{X}^N(t)-u(\cdot,t)\Vert_\infty = 0\,, \quad \mbox{almost surely}
\]
if $b$ has compact support.
 
For general $b$, we consider 
\begin{equation}\label{max}
M\;=\;\sup_{t\in[0,T]} \Vert u(\cdot,t)\Vert_\infty \,,
\end{equation}
which is finite since we are imposing $T<T_{\textrm{max}}$.

Let $b_{M+1}$ be a smooth function with compact support that coincides with $b$ in the interval $[0,M+1]$. Denote by $X^{N,M+1}(x,t)$ the process with creation of particles driven by $b_{M+1}$ instead  of $b$. 

By making use of Harris graphical construction of Section
\ref{graphical.construction}, for each $N$ we can couple the processes
$X^{N,M+1}(x,t)$ and $X^{N}(x,t)$ in such a way that their trajectories coincide
up to the stopping time
\begin{equation*}
 T^N_{M+\frac12}\;=\; \inf\{t\geq 0;\; \Vert X^{N, M+1}(t)\Vert_\infty \ge
M+\pfrac12\}.
\end{equation*}
Observe that
\begin{equation}\label{eq20}
\Vert X^N(t)- u(\cdot,t)\Vert_\infty \le \Vert X^N(t)- X^{N, M+1}(t)\Vert_\infty
+ \Vert X^{N, M+1}(t)-u(\cdot,t)\Vert_\infty
\end{equation}
Denote by $u^{M+1}(\cdot,t)$ the solution of \eqref{PDE}  with  $f=b_{M+1}-d$.
Since uniqueness hold for \eqref{PDE} and $\|u(\cdot,t)\|_\infty\le M$ in
$[0,T]$, we have
$u^{M+1}(\cdot,t)=u(\cdot,t)$, for $t\in[0, T]$.
Thus, the second term in \eqref{eq20} can be replaced by $\Vert X^{N,
M+1}(t)-u^{M+1}(t)\Vert_\infty$, which goes to zero as $N\uparrow \infty$
uniformly and almost surely. Hence there exists a finite random $N_0$ such that,
for $N\ge N_0$,
\begin{equation*}
 \sup_{t\in[0,T]}\Vert X^{N, M+1}(t)\Vert_\infty\; \leq\;  \sup_{t\in[0,T]}\Vert
u(\cdot,t)\Vert_\infty+\pfrac12\;= \; M+\pfrac12\,.
\end{equation*}
Since $X^N$ and $X^{N,M+1}$ are coupled, the first term on the r.h.s of
\eqref{eq20} vanishes for $N \ge N_0$, concluding the proof of the theorem.
\end{proof}

Observe that for $\eps>0$,  Theorem \ref{t2.1} implies $\liminf_{N\to\infty}
T_{\textrm{expl}}(N)\geq T_{\textrm{max}}-\eps$ almost surely.  Letting
$\eps\searrow 0$ yields
\begin{equation*}\label{eq25}
 \liminf_{N\to\infty} T_{\textrm{expl}}(N)\;\geq\; T_{\textrm{max}} \,,\quad \textrm{almost surely},
\end{equation*}
proving Corollary \ref{cor.liminf}.

\section{The explosion times}
\label{asymp.behavior}

In this section we prove the finiteness of the explosion time $\Tm^N$ for every
$N$, Proposition \ref{bu.prob.1} and the convergence as $N\to \infty$, Theorem
\ref{t2.2}. Both proofs rely on a coupling with a one-dimensional birth and
death process.

The goal is to construct (jointly with $X^N$) a process $Y=(Y(t), t\ge 0)$ that
dominates the total number of particles $\|\eta(t)\|_\infty:=\sum_k \eta_k(t)$
from below
almost surely and for every time. According to the model, the rate at which a
new particle (somewhere) is created is given by $\sum_{k\in \T_N} \ell
b(\frac{\eta_k}{\ell})$, the sum of the rates of each individual site. Since $b$
is convex and nondecreasing,
\begin{equation}\label{eq27}
   \sum_{k\in \T_N} \ell b\Big(\frac{\eta_k}{\ell}\Big) \; =\; \ell N \sum_{k\in \T_N} \frac{1}{N} b\Big(\frac{\eta_k}{\ell}\Big)
\;\geq\; \ell N b\Big(\frac{1}{\ell N}\sum_{i\in \T_N} \eta_i\Big)\;\geq \ell N b\Big(\frac{|\eta|}{\ell N}\Big) =: q(|\eta|, |\eta|+1)\,.
\end{equation}
Analogously, if $d$ is bounded, the rate for annihilation of a particle is the
sum of the rates in each site, which we bound (for $\eta\ne0$) by
\begin{equation}\label{eq28}
   \sum_{k\in \T_N} \ell d\Big(\frac{\eta_k}{\ell}\Big) \; \leq \ell N \Vert d \Vert_\infty\, =: q(|\eta|,|\eta|-1)\,.
\end{equation}
Then, we can construct a process $(Y(t))$ with rates $q$ jointly with
$(\eta(t))$, in such a way that  $\|\eta(t)\|_\infty\ge Y(t)$ almost surely. We
need to
slightly modify the construction of $\eta(t)$. The construction is almost the
same with the only modification that for each $k$, we construct the Poisson
processes $\Nd(r, k),\, r\ge 1$ as independent thinnings of a Poisson process
$\Nd(k)$ with rate $\ell \|d\|_\infty$ and we add (independent) uniform marks to
the points of the Poisson processes $\Nb(r,k)$, i.e.: to each point $t\in
\Nb(r,k)$ we attach a random variable $U_t$ with uniform distribution in
$[0,1]$. All of them independent of all the processes. We construct $Y$
recursively. Assume that the process is defined up to time $s$. Let
\[
 \tau(s, Y(s)):=\inf\left\{t>s \colon t\in \cup_{k\in \T_N} \Nb(\eta_k(s),k)
\cup \Nd(k) \right\}\,.
\]
For $t\in [s,\tau(s,Y(s)))$ define $Y(t)=Y(s)$ and then
\begin{enumerate}
 \item[1.] If $\tau \in \cup_{k\in \T_N} \Nb(\eta_k(s),k)$ and $U_{\tau(s,Y(s))}
< \ell N b\Big(\frac{|\eta|}{\ell N}\Big)\Big/\sum_{k\in \T_N} \ell
b\Big(\frac{\eta_k}{\ell}\Big)$, set $Y(\tau)=Y(s)+1$.
 \item[2.] If $\tau \in \cup_{k\in\T_N}\Nd(k)$, set $Y(\tau)=Y(s)-1$.
 \item[3.] Otherwise set $Y(\tau)=Y(s)$.
\end{enumerate}
The process is then defined up to time $\tau$. Put $s=\tau$ and iterate to
define $Y(t)$ as long as possible. This construction guarantees that if for some
time $s\ge0$ we have $\|\eta(s)\|_\infty\ge Y(s)$ then
\begin{equation}\label{eq291}
\|\eta(t)\|_\infty\geq Y(t)\,, \qquad \textrm{for all }  t\geq s\,.  
\end{equation}
It is straightforward to check that in addition, $Y$ is a birth and death process with rates $q(r,r+1)= \ell N b(r/\ell N)$, and $q(r,r-1)=\ell N \|d\|_\infty$.
The next step is to proof that $Y$ explodes in finite time and estimate the explosion time of $Y(t)$. We begin with some considerations about birth and death processes with rates $\mathbf{b}_r$ and  $\mathbf{d}_r$ respectively.
For such a process, let $\tau_r$ be the hitting time of the state $r\in \bb N$
and denote
$$f_r\;=\; \bb E_r[\tau_{r+1}]\,,$$
the expected time to hit $r+1$, starting at $r$. 
By the Markov property,  
$$ f_{r+1} = \frac{1}{\mathbf{b}_r+\mathbf{d}_r}+ \frac{\mathbf{d}_r}{\mathbf{b}_r}f_r\,,\quad \textrm{ for }r\geq 0\,.$$
Notice that 
$\sum_{r=r_0}^\infty f_r$ is the expected time spent by the process starting from $r_0\in \bb N$ before reaching $\infty$. With an inductive procedure we derive the formula
\begin{equation}\label{eq_34}
 \sum_{r=1}^n f_r\;=\; \sum_{r=0}^{n-1}\frac{1}{\mathbf{b}_r+\mathbf{d}_r}+ f_0  \sum_{r=0}^{n-1} \prod_{j=0}^r \frac{\mathbf{d}_j}{\mathbf{b}_j}+
\sum_{r=0}^{n-2} \frac{1}{\mathbf{b}_r+\mathbf{d}_r} \sum_{j=r+1}^{n-1}\prod_{i=r}^j \frac{\mathbf{d}_i}{\mathbf{b}_i}\,.
\end{equation}
We invoke now that $\mathbf{b}_r=\ell N b(\frac{r}{\ell N})$, $\mathbf{d}_r=\ell N \Vert d \Vert_\infty$. Notice also that
$$
\int_1^\infty \frac{ds}{b(x)}\;<\;\infty\,.
$$
If $d$ is bounded (or linear), then $\mathbf b_r/\mathbf d_r \to 0$ as $r \to
\infty$ and hence the three terms in \eqref{eq_34} are finite as $n \to \infty$.
 The first one behaves as $\int 1/b$, each term of the sum in the second term
can be bounded by $Ce^{-r}$ and using these two facts, we bound the third term.
We have proved Proposition \ref{bu.prob.1} for $d$ bounded. Observe that if $d$
is not bounded but linear, equation \eqref{eq28} takes the form
\[
\sum_{k\in \T_N} \ell d\Big(\frac{\eta_k}{\ell}\Big) \; = d(|\eta|)\, =:
q(|\eta|,|\eta|-1)\,.
\]
Hence we don't need the thinnings to couple $Y$ and $\eta$. Both processes
move jointly to the left always. The rest of the proof follows along the same
steps.

The proof of Theorem \ref{t2.2} is based on a more delicate analysis of \eqref{eq_34}, but we first need some knowledge on the solutions of \eqref{PDE}. That is the context of the following theorem.

\begin{theorem}\label{explosion_PDE} Assume $f$ is nonnegative and continuously differentiable. 
\begin{enumerate}

\item There exists a time $\Tm>0$ (possibly infinite)  such that
there exists a unique maximal solution to \eqref{PDE} in $[0,\Tm)$. 

\item If $\Tm<+\infty$ we have
\[
 \lim_{t\nearrow \Tm} \|u(t,\cdot)\|_{L^\infty(\T)}  = +\infty\,.
\]
\item If $f$ is convex and positive for $r\ge r_0$ and $\varphi(x) \ge
r_0$ for all $x \in \T$, then $u$ blows up in finite time $\Tm$ and moreover the
following estimate holds
\[
 \Tm\le \int_{\|\varphi\|_{L^1(\T)}}^\infty \frac{1}{f(s)}\,ds\,.
\]
\item If $f(u)=u^p$ with $1< p\le 3$, then 
\[
   \lim_{t\nearrow \Tm} \|u(t,\cdot)\|_{L^1(\T)} = +\infty\,.
\]
\end{enumerate}
\end{theorem}
\begin{proof}
Existence and uniqueness of a smooth classical (maximal) solution up to time
$\Tm$ is proved in \cite{Weissler}. See also \cite{Chen, Pao, QS, SGKM}.

To prove (2) consider $y$, the solution of the ODE $y'=f(y)$,
$y(0)=\|\varphi\|_{L^\infty(\T)}$. By comparison arguments, we get $u(x,t)\le
y(t)$ for every $t$ and hence the blow-up time $\Tm=\Tm(\varphi)$ is bigger than
the one for $y$. So, assume
\[
 \limsup_{t\nearrow \Tm} \|u(t,\cdot)\|_{L^\infty(\T)} \le C\,.
\]
Then, there is an increasing sequence of times $(t_k)$ such that $t_k \nearrow
\Tm$ and such that the solution to \eqref{PDE} with initial data $u(\cdot, t_k)$
is defined in some interval $[0,\tilde T]$, where $\tilde T$ depends only on
$C$ (and not on $k$). Then we can extend $u$ for times $t>\Tm$ which contradicts
(1).

For (3) consider $\Phi(t):=\|u(\cdot, t)\|_{L^1(\T)}$. Observe that
$v(x,t)=\varphi(x)$ verifies $v_t \le v_{xx} + f(v)$  and then the solution $u$
with initial data $\varphi$ verifies $u(x,t)\ge v(t,x)= \varphi(x)$ for every
$x\in \T$, $t\ge 0$. Differentiating and using Jensen's
inequality we get
\[
 \frac{d}{dt}\Phi(t) = \int_{\T} f(u(x,t))\, dx \ge f(\Phi(t))\,, \qquad
\Phi(0)={\|\varphi\|_{L^1(\T)}}\,.
\]
Thus, $\Phi(t) \ge z(t)$, the solution of the ODE $z'=f(z),
z(0)=\|\varphi\|_{L^1(\T)}$. Integrating this equation we obtain
\[
\int_{z(0)}^{z(t)} \frac{1}{f(s)}\,ds \ge t\,.
\]
Let $T_z$ be the maximal existence time for $z$. Taking the limit $t\nearrow T_z$,
we first observe that $T_z <\infty$ and $z(T_z)=+\infty$, and next
\[
 \Tm \le T_z \le \int_{\|\varphi\|_{L^1(\T)}}^\infty \frac{1}{f(s)}\,ds\,.
\]
For (4) we use the following dichotomy proved by Vel\'azquez
\cite{Velazquez}. Assume that $u(x,t)$ is a positive solution to $u_t=u_{xx} +
u^p$ for $x\in (-R,R)$ and $t \in (0,T)$ which blows up at $t=T$. Assume also
that its blow-up set 
\[
B=\{\bar x \in (-R,R) \colon \limsup_{t\nearrow T}
\|u(t,\cdot)\|_{L^\infty(\T)}  = +\infty\}
\]
is contained in $[-\delta, \delta]$
for some $\delta <R$. Then $B$ is isolated and for any blow-up point $\bar x \in
B$ one of the following holds.
\[
\begin{split}
& \lim_{x\to \bar x} \left(\frac{|x-\bar x|^2}{|\log|x-\bar
x||} \right)^{\frac{1}{p-1}}u(x,\Tm) = \left(
\frac{8p}{(p-1)^2}\right)^\frac{1}{p-1},\\
&\lim_{x\nearrow\bar x}|x-\bar x|^{\frac{m}{p-1}}u(x,\Tm)=((p-1)
C)^{-\frac{1}{p-1}}\,,
\end{split}
\]
where $C$ and $m$ are positive constants with $m\ge 4$. Also Chen and Matano
proved \cite{Chen-Matano} that if the initial data $\varphi$ is non-constant,
the number of blow-up points of \eqref{PDE} is finite (moreover, it does not
exceed the number of local maximum of $\varphi$) and hence we can
apply the above dichotomy in a neighborhood of a blow-up point. Fatou's lemma and straightforward
computations leads to $ \liminf	_{t\nearrow T} \|u(\cdot,t)\|_{L^1(\T)} \ge
\|u(\cdot,T)\|_{L^1(\T)} = +\infty$ if $1<p\le 3$ for any of the alternatives.
If $\varphi$ is constant, the conclusion is immediate and holds for every $p>1$.
\end{proof}

\begin{proof}[Proof of Theorem \ref{t2.2}]
If $\Tm = \infty$, Corollary \ref{cor.liminf} implies the theorem. We observe
that this only happens if $f$ has a root at $r_0$ and $\varphi \equiv
r_0$. For $\Tm<\infty$, we assume $d$ is bounded (if $d$ is not bounded but
linear, the proof is similar). From \eqref{eq_34} we get
\begin{equation}\label{bound.bu}
\begin{split}
 \sum_{r=y\ell N}^n f_r\;= & \; \sum_{r=y\ell N}^{n-1}\frac{1}{\mathbf{b}_r+\mathbf{d}_r}+ f_0  \sum_{r=y\ell N}^{n-1} \prod_{j=0}^r \frac{\mathbf{d}_j}{\mathbf{b}_j}+
\sum_{r=y\ell N}^{n-2} \frac{1}{\mathbf{b}_r+\mathbf{d}_r} \sum_{j=r+1}^{n-1}\prod_{i=r}^j \frac{\mathbf{d}_i}{\mathbf{b}_i}\\
\le \;  & \frac{1}{\ell N}\sum_{r=y\ell N}^{n-1}\frac{1}{b(r/\ell N)} + f_0 C  \sum_{r=y\ell N}^{n-1} e^{-r}+
\frac{1}{\ell N}\sum_{r=y\ell N}^{n-2} \frac{C}{b(r/\ell N)} \sum_{j=r+1}^{n-1}e^{-(j-r)}\\
\le \; & \int_{y}^{\infty}\frac{1}{b(s)} \, ds + C' e^{- \frac{y\ell N}{2}}+ C''\int_{y}^{\infty}\frac{1}{b(s)} \, ds\,.
\end{split}
\end{equation}
Since there exists $r_0$ such that $\mathbf{d}_r/\mathbf{b}_r < e^{-1}$ for $r
\ge r_0$ and the indexes in the sums start at $r=y\ell N$, the constants $C,\,
C', \, C''$ are independent of $\ell$ and $N$.

% . For the reverse inequality we restrict ourselves to the case  $b(s)=s^p-d(s)$, $1<p\leq 3$.
% Let $(Y_t)$ be a one-dimensional birth-and-death process on $\bb N$ with jump rates $q(i,j)$ given by 
% \[q(i,j)=
% \begin{cases}
% \ell N b\Big(\frac{i}{\ell N}\Big), &  i\ge 0, \, j=i+1,\\
% \ell N \Vert d\Vert_\infty, &  i>0, \, j=i-1,\\
% 0 & \mbox{otherwise}
% \end{cases}
% \]
Fix  $\delta>0$ and choose $M$ large enough in order to guarantee that
\[
 \int_{M}^{\infty}\frac{1}{b(s)} \, ds + C' e^{- \frac{M}{2}}+ C''\int_{M}^{\infty}\frac{1}{b(s)} \, ds <\delta\,.
\]
Observe that $M$ does not depend on $N$. By  Theorem \ref{explosion_PDE},  there exists a time
$T<T_{\textrm{max}}$ such that
$$\Vert u(\cdot,T)\Vert_{L^1(\T)}\geq M+1\,.$$
Take $\eps=1$ and apply Theorem \ref{t2.1}, to get the existence of a finite random $N_0\in \bb N$ such that, for $N\geq N_0$, 
$\|X^N(T,\cdot)-u(T,\cdot)\|_\infty < 1$, which implies $\|X^N(T,\cdot)\|_{L^1(\T)}>M$. Hence,
$$\sum_{k\in \T_N} \eta_k(T)\geq M\ell N \,.$$
For each $N$ and from time $T$ on we construct the process $Y$, with initial data $Y(T)=M \ell N$ jointly with $\eta(t)$. Using that the expected time for the explosion time of the process $Y$ is given by
\[
T +  \sum_{r=M\ell N}^\infty f_r
\]
and \eqref{bound.bu} we  get that the the explosion time of $Y$, $T_Y$, has
expectation bounded by $T + \delta$ and that for $N\ge N_0$, this time bounds
from above the explosion time $\Tm^N$ of $\eta$. Hence, for $\gamma >0$ we have
\[
\begin{split}
 \P(\Tm^N > \Tm + \gamma) \le &   \,\P(\Tm^N > \Tm + \gamma, N > N_0) + 
\P(\Tm^N > \Tm + \gamma, N \le N_0)\\
 \le & \, \P(T_Y > T + \gamma) + \P(N\le N_0)\,.
\end{split}
\]
Hence, the finiteness of $N_0$ and Markov inequality gives us
\[
 \limsup_{N\to \infty} \P(\Tm^N > \Tm + \gamma) \le \frac{\delta}{\gamma}\,.
\]
Since $\delta$ is arbitrary and Corollary \ref{cor.liminf} implies the
reversed inequality, the proof is completed.

\end{proof}

% \section{Open questions}
% \patu{No me convence mucho esta seccion...}
% We end this paper with some questions that we are not able to answer yet.
% \begin{enumerate}
% \item[(i)] As a function of $b$ and $d$, how many sites do explode in the
% particle system? Taking $b(x)=x^{2+\eps}$, $d(x)=x$ and using 
% the Subsection \ref{BD},
% it is possible to prove that, almost surely,
% only one site explodes. The authors expect that the numbers of sites that
% explode should be a deterministic function of $b$ and $d$. \patu{Podemos de
% verdad??}  
% \item[(ii)] In the case of blow-up of a single site just as in the item
% before, the position of this site does converge to the macroscopical point
% where the PDE's solution explodes in the continuum?
% \item[(iii)] Considering the Zero Range process instead of independent random
% walks, is it true the convergence of Theorem \ref{t2.1}? We notice that the
% method here presented does not apply due the inexistence of a suitable
% non-linear Duhamel's Principle.
% \end{enumerate}

\appendix

\section*{Acknowledgements}
We want to thank Pablo Ferrari, Milton Jara and Mariela Sued for fruitful
discussions.

PG is partially supported by UBACyT 20020090100208, ANPCyT PICT No.
2008-0315 and CONICET PIP 2010-0142 and 2009-0613.
\bibliography{bibliografia}
\bibliographystyle{plain}

\end{document}